\documentclass[12pt]{amsart}
\usepackage{amssymb,amsthm,amsmath,listings,multicol,mathtools}
\usepackage[headings]{fullpage}
\usepackage{amsfonts,graphicx,transparent,enumitem}
\usepackage{mdframed,paracol}
\usepackage[dvipsnames]{xcolor}
\usepackage{color,bm}
\DeclareMathOperator{\lcm}{lcm}
\usepackage{multicol}
\newcommand{\twocols}{\begin{multicols}{2}}
\newcommand{\onecol}{\end{multicols}}
\newcommand{\threecols}{\begin{multicols}{3}}
\newcommand{\fourcols}{\begin{multicols}{4}}

\usepackage[all]{xy}
\usepackage[normalem]{ulem}
\usepackage[
  colorlinks=true,
  linkcolor=blue,
  citecolor=blue,
  urlcolor=blue
]{hyperref}

\usepackage{framed}

\newtheorem{Theorem}{Theorem}

\newtheorem{Algorithm}[Theorem]{Algorithm} 
\newtheorem{Lemma}[Theorem]{Lemma} \newtheorem{Proposition}[Theorem]{Proposition} \theoremstyle{definition}
  \newtheorem{Example}[Theorem]{Example}  \newtheorem{Remark}[Theorem]{Remark}  \newtheorem{Question}[Theorem]{Question}  
   
   \numberwithin{equation}{section}   
 \newtheorem*{Conjecture*}{Main Conjectures}

\def\<{\langle}
\def\>{\rangle}
\title{Betti Tables and the Licci Property}
\author{Adam Boocher}
\address{Department of Mathematics, University of San Diego}
\email{aboocher@sandiego.edu}

\begin{document}
\begin{abstract}
We show that the licci property of an ideal cannot, in general, be determined from its Betti table by exhibiting two $\mathfrak{m}$-primary monomial ideals containing the cubes the variables that have identical Betti tables, yet only one is licci.  In contrast, we show that for monomial ideals containing the square of each variable, the licci property is completely determined by the Betti table; indeed, it is determined by its first two columns. 
\end{abstract}

\maketitle
\vskip -.3in 
\
\section{Introduction}
\noindent 

Let $R$ be a $g$-dimensional polynomial ring over a field $k$, and let $\mathfrak m$ be the ideal generated by the variables. We say that an ideal $I$ is licci if it is in the linkage class of a complete intersection, e.g. see \cite{HU,HUMonomial,peskine,rao}.  The following question has been asked by Huneke, Polini, and Ulrich. 

\begin{Question}\label{Q1}
When can the licci property of an ideal $I$ be determined from the Betti table of $R/I$?
\end{Question}

The first result of this paper is an example showing that in general the Betti table cannot determine the licci property.  

\begin{Example}\label{counterexample}
Consider the two ideals of $k[a,b,c]$, 
$$I = \left(a^{3},b^{3},c^{3},bc^{2},a^{2}bc,a^{2}b^{2}\right),\ J = \left(a^{3},b^{3},c^{3},bc^{2},a^{2}c^{2},a^{2}b^{2}\right).$$  The quotient rings $R/I$ and $R/J$ each have the same Betti table: 
$$\begin{array}{r|ccccc}
      & 0 & 1 & 2 & 3\\ \hline 
     0: & 1 & . & . & .\\
     1: & . & . & . & .\\
     2: & . & 4 & 1 & .\\
     3: & . & 2 & 6 & 1\\
     4: & . & . & 1 & 2
     \end{array}.$$	
 The ideal $I$ is licci whereas $J$ is not. This can be seen, for instance by using the algorithm from \cite{HUMonomial} described below. These examples were discovered via an extensive computer search.
\end{Example}
To the best of our knowledge, this is the first such example in the literature, although there are many ``near'' examples, e.g. in \cite{HMNU}.  Note that both ideals in this example are $\mathfrak{m}$-primary as they contain the cube of each variable.  
     There is an algorithm in \cite{HUMonomial} which determines whether a given $\mathfrak{m}$-primary monomial ideal is licci. In that paper the authors showed that any ideal generated by at most 5 monomials must be licci, thus with regard to Question 1, our example has the smallest possible number of generators. There is another sense in which this example is ``smallest'', for if $I$ is monomial and contains the square of each variable, then the answer to Question 1 is yes, which is the main result of this paper. 
     
     \begin{Theorem}
     If $I$ is a monomial ideal that contains the square of each variable, then the licci property of $I$ can be detected from the Betti table of $R/I$. 	
     \end{Theorem}

	In fact, licci monomial ideals containing a square of the variables must have a particularly nice form and the combinatorics among their generators is tight enough that the licci property can be determined by knowing only the degrees of the generators and first syzygies.

\begin{Theorem}\label{MainThm} Let $R$ be a $g$-dimensional polynomial ring over a field $k$ and suppose $I$ is a monomial ideal in $R$ that contains the square of each variable. Suppose that the squarefree generators of $I$ lie in degrees $2\leq i\leq r$ for some integer $r$.  Then the licci property of $I$ can be detected from the Betti numbers $\beta_{1j}$ and $\beta_{2,{j+1}}$ for $2 \leq j \leq r$.

More explicitly, suppose that for each $j$, $I$ has $b_j$ squarefree minimal generators. Then the following are equivalent: 

\begin{enumerate}
\item $I$ is licci;
\item For all $2\leq j \leq r$ the following hold:
$$\begin{array}{rclclclcl}
	\beta_{23} &= &{b_2 \choose 2} & &&
	+ & 2b_2  \\
	\beta_{24} &=& {b_3 \choose 2} + b_3b_2 & 
	+ & {g \choose 2} - b_2 &
	+ & 3b_3 & 
	+ & (g-b_2-1)b_2\\
	\beta_{25} &=& {b_4 \choose 2} + b_4(b_2+b_3) & 
	+ & &
	+ & 4b_4 & 
	+ & (g-b_2-b_3-2)b_3\\
	& \vdots & \\ 
	\beta_{2,j+1} &=& {b_j \choose 2} + b_j(b_2+\cdots+b_{j-1}) & 
	+ & &
	+ & jb_j & 
	+ & (g-(b_2+ \cdots +b_{j-1}+(j-2)))b_{j-1}\\
\end{array}$$
\item The squarefree generators of $J$ in each degree are, up to relabeling: 
\begin{eqnarray*}
	& & a_1(y_1, \ldots, y_{b_2}),\\	
	& & a_1a_2(z_1, \ldots, z_{b_3}),\\	
	& & \vdots \\
	& & a_1a_2\cdots a_{r-1}(\omega_1, \ldots, \omega_{b_r}),
\end{eqnarray*}
where all the indexed letters denote distinct variables. 
\end{enumerate}
\end{Theorem}

As an illustration of the theorem, below are the Betti tables of two $\mathfrak{m}$-primary monomial ideals in 6 variables that contain the square of each variable and 4 additional cubic generators. Theorem \ref{MainThm} says that the corresponding ideals will be licci if and only if $\beta_{34} = 33$.   See Example \ref{Example8} for a slightly larger example.
\begin{center}
\includegraphics[width=4in]{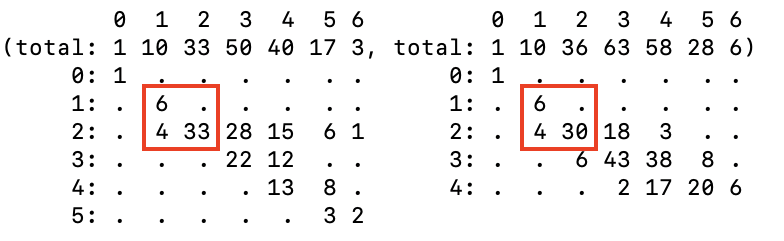}
\end{center}

Evidently, when the monomial ideal contains the squares of the variables, the licci property is strong enough to force this strict structure.  However, given the relatively simple ideals in Example \ref{counterexample}, once the degree grows, there cannot be such a characterization.  It is interesting to ask whether this behavior can be generalized.

\begin{Question}
Outside the monomial case, if $I$ contains a maximal regular sequence of quadrics, can the licci property be detected from the Betti table of $R/I$?  
\end{Question}

\begin{Question}
What other squarefree monomial ideals $J$ have the property that their structure (up to a relabeling of the variables) can be determined completely from knowing the first two columns of the Betti table of $(x_1^2,\ldots, x_g^2) + J$? 
\end{Question}

The proof of Theorem \ref{MainThm} uses an algorithm in \cite{HUMonomial} which determines whether an $\mathfrak{m}$-primary monomial ideal $I$ is licci.   Once familiar with the algorithm it will be clear that (3) $\iff$ (1).  A careful calculation with the Taylor resolution will show that $(3) \implies (2)$.  Finally a combination of the earlier methods will be used to complete the proof and show that $(2) \implies (3)$.

\medskip 

\noindent {\bf Notation:} For the rest of the paper, $R$ will denote a polynomial ring of dimension $g$ and $I\subset R$ will be a monomial ideal that contains the squares of the variables.  The ideal generated by the squares of the variables will be denoted by $Q$, and we will write $I = Q + J$, where $J=(G)$ is a squarefree monomial ideal whose set of minimal monomial generators is $G$.

\section{The Equivalence of (1) and (3)}\label{sec:2}
This section summarizes and uses the Algorithm in \cite{HUMonomial} to establish part of our main theorem.   Note that their algorithm applies to any ${\mathfrak m}$-primary ideal.  In the case we consider here, where $I$ contains the squares of the variables, the algorithm becomes slightly simpler than in the general case.  We state that simpler version. 

\begin{Algorithm}[Huneke-Ulrich Theorem 2.6 in \cite{HUMonomial}]\label{alg1}
Consider the following algorithm ${\bf A(G)}$ which takes input $G$, a set of squarefree monomials, and returns {\bf True} or {\bf False} as follows:
\begin{itemize}
	\item If $|G|\leq 1$, return {\bf True.}
	\item Else, if the gcd of the elements of $G$ is $1$ then return {\bf False}. 
	\item Else, let $h= \gcd(G)$ and define $$G' =\{g/h \ | \ g\in G \mbox{ and }\deg (g/h) \geq 2\}.$$ (We have divided all the elements by their gcd and removed any pure variables that remain).  Return $\mathbf{A(G')}$. 
	\end{itemize}
Let $I$ be a monomial ideal containing the square of each variable and let $G$ be the set of squarefree minimal generators. Then $I$ is licci if and only if the algorithm above applied to $G$ returns {\bf True}.
\end{Algorithm}

\begin{Example}\label{Example8}  Suppose $I = Q + (ab,acd,acef)$ 
in the ring $k[a,b,c,d,e,f].$  
Then in the algorithm we would begin by noting $G = a\{b,cd,cef\}$.  We'd then apply the algorithm to $G' = \{cd,cef\} = c\{d,ef\}$ and then finally to $G'' = \{ef\}$ to conclude that $I$ is licci. Note that the ideal $(G)$ can be written as 
$$\textcolor{blue}{a}(b) + \textcolor{blue}{a}\textcolor{red}{c}(d) + \textcolor{blue}{a}\textcolor{red}{c}\textcolor{orange}{e}(f).$$
In the notation of Theorem \ref{MainThm}, we have that $g = 6$ and $(b_2,b_3,b_4) = (1,1,1)$.  Part (2) asserts that the entries in the second column of the Betti table should be $2,22,8$. The Betti table of $I$ is given below (on the left) and we can see that these numbers match. 

If instead, $I = Q + (ab,acd,acef,acxy)$ (now in a ring with $g = 8$ variables) then the algorithm would begin as in the first example, but would conclude with calculating ${\bf A(G'')}$ with $G'' = \{ef,xy\}$ and thus $I$ is not licci. Here we have $(b_2,b_3,b_4) = (1,1,2)$ and we can see that the formula would predict that $\beta_{2,5} = \textcolor{red}{17}$, instead of the true value of $16$.  The Betti tables are respectively given below.

\begin{center}
\includegraphics[width=5in]{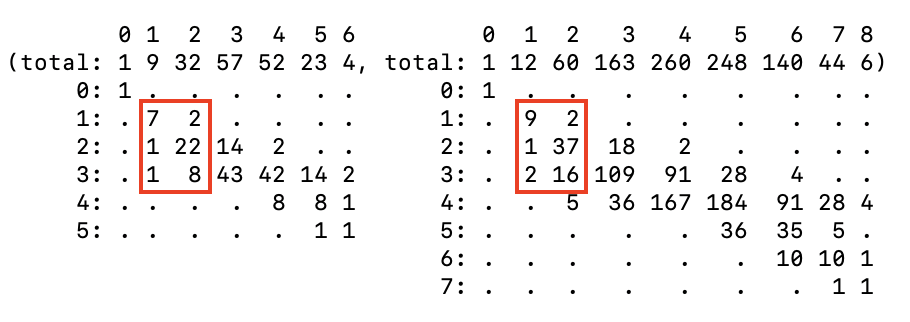}
\end{center}
\end{Example}
As a running example throughout the paper, we will consider the special case when all the generators have degree $2$. 

\begin{Example}\label{ex6}
Suppose that $I =Q + (G)$ where $G$ consists of quadratic squarefree monomials. Let $J = (G)$.  Consider what Algorithm \ref{alg1} says. If $J$ has height 2 or more, then $I$ is not licci.  However if $J$ has height $1$, then evidently $J = a(y_1, \ldots, y_n)$ for some distinct variables $a, y_i$, and thus $I$ is licci.  We have the following special case of the equivalence between $(1)$ and $(3)$ in Theorem \ref{MainThm}.  
\begin{center}
	$I$ is licci if and only if $J = \textcolor{blue}{a}(y_1, \ldots, y_n)$ for some variables $a, y_i$.
\end{center}
\end{Example}

\begin{Lemma}\label{lemma:samedegree}
Suppose that $G$ is a set of squarefree monomials.  Suppose that the minimal degree of a monomial in $G$ is $d\geq 2$. Let $h = \gcd(G)$.  Suppose that in Algorithm \ref{alg1}, ${\bf A}(G)$ returns {\bf True}.  Then 
$$\deg h = \begin{cases}
 	d-1 & \mbox{if $|G| \geq 2$ }\\
 	d & \mbox{if $|G|=1$} \\ 	
 \end{cases}.$$
Thus if $|G| \geq 2$, the minimal generators of smallest degree $d$ are of the form $\{hy_1, hy_2, \ldots, hy_m\}$ for some distinct variables $y_1, \ldots, y_m$.
\end{Lemma}
\begin{proof}
The claims are clear if $|G| = 1$, so we suppose $|G|\geq 2$.  Since ${\bf A(G)}$ returns {\bf True}, the first step of Algorithm \ref{alg1} will extract $h$ and then we will need to calculate 
$$G' =\{g/h \ | \ g\in G \mbox{ and }\deg (g/h) \geq 2\}.$$
We argue by contradiction.  Suppose $\deg h < d-1$. Then for each $g\in G$ we will have $\deg (g/h) \geq 2$ and so $G' = \{g/h \ | \ g\in G \}$ and note that these monomials have gcd 1. But then ${\bf A(G')}$ would return {\bf False}, a contradiction.
\end{proof}
\begin{Example}\label{ExDecomp}
It is helpful to see a general application of Algorithm \ref{alg1}.  Suppose that $I$ is an ideal containing the squares of the variables and whose other generators are: 4 cubics, 7 quintics, and 13 sextics, that is $b_{3} = 4$, $b_{5} = 7$, $b_{6} = 13$. Suppose that $I$ is licci.

By Lemma \ref{lemma:samedegree} we have that the 4 cubic generators must be the generators of the ideal  $\textcolor{red}{a_1a_2}(y_1,y_2, y_3,y_4)$ (for distinct variables $a_i,y_j$). The algorithm will then continue, applied to the set $G'$ after we drop the remnants of the cubics and divide the higher order terms by $\textcolor{red}{a_1a_2}$.  $G'$ consists now of monomials of degrees $5-2=3, 6-2=4$. Now since ${\bf A(G')}$ still returns {\bf True}, again by Lemma \ref{lemma:samedegree} we have that these smallest degree generators of $G'$ must have the form $\textcolor{orange}{a_3a_4}(z_1\ldots, z_7)$, again for some distinct variables $a_i, c_j$, and then one final application of ${\bf A}$ to the remaining $13$ (now) quadratic generators will show us that our original generators must have been of the form: 
$$Q + \textcolor{red}{a_1a_2}(y_1, \ldots, y_4) + \textcolor{red}{a_1a_2}\textcolor{orange}{a_3a_4}(z_1, \ldots, z_7)  + \textcolor{red}{a_1a_2}\textcolor{orange}{a_3a_4}\textcolor{blue}{a_5}(w_1,\ldots, w_{13}).
$$
\end{Example}

The following theorem gives the general case of this phenomenon and follows from Lemma \ref{lemma:samedegree} via induction.
\begin{Theorem}\label{Thm5}
	Suppose that $I$ is a licci monomial ideal containing the squares of the variables. Suppose that $I$ has $b_j\geq 0$ squarefree generators for $2\leq j \leq r$, (where $b_r>0$). Then there exist distinct variables $a_1,\ldots, a_{r-1}$ and (still distinct) variables $y_{ij}$ 
	such that the squarefree generators of $I$ in degree $j$ are precisely the monomial generators of the ideal $a_1\cdots a_{j-1}(y_{j1}, y_{j2}, y_{j3}, \ldots, y_{jb_j})$.
	
	Conversely, any ideal (containing the square of each variable) whose squarefree generators have this form is licci. 
\end{Theorem}

This theorem completes the proof that $(1)$ and $(3)$ are equivalent in Theorem \ref{MainThm}.

\section{The Taylor Resolution and the proof that (3) $\implies$ (2)}\label{sec3}

Given such a tight characterization of licci ideals, now all that remains is to show the relationship between this structure and the Betti numbers $\beta_{1j}(R/I)$ and $\beta_{2j}(R/I)$.  We will use the Taylor resolution, which we review below. For more details see, e.g. \cite{taylorthesis, millersturmfuls,peevaBook,merminThree}. 

\medskip 

\noindent {\bf The Taylor Resolution -}
Let $M$ be a monomial ideal whose minimal generators are $f_1, \ldots, f_b$.  Let $F_1, F_2, F_3$ be free modules with bases $e_i$, $e_{ij}$, $e_{ijk}$ respectively where $i<j<k \in \{1, \ldots, b\}$ are distinct. The Taylor resolution begins with with following exact sequence: 
$$
\xymatrix{
F_3 \ar[r]^{\gamma} & F_2 \ar[r]^{\psi} & F_1 \ar[r]^\phi & R \ar[r] & R/M
}
$$
whose maps are basically the maps in the Koszul complex, after dividing by the natural common factors: 
\begin{eqnarray*}
	\phi(e_i) &=& f_i\\
\psi(e_{ij}) &=& \frac{\lcm(f_i,f_j)}{f_i} e_i - \frac{\lcm(f_i,f_j)}{f_j} e_j\\
\gamma(e_{ijk}) &=& \frac{\lcm(f_i,f_j,f_k)}{\lcm(f_i,f_j)} e_{ij} 
-\frac{\lcm(f_i,f_j,f_k)}{\lcm(f_i,f_k)} e_{ik}
+\frac{\lcm(f_i,f_j,f_k)}{\lcm(f_j,f_k)} e_{jk}.
\end{eqnarray*}

When we refer to ``the Taylor syzygy on the pair $(f_i,f_j)$'' we mean the element $\psi(e_{ij})$.  
In general the Taylor resolution is not minimal. However, note that the matrices for $\phi, \psi$ do have entries in the maximal ideal, so any potential non-minimality will be from the map $\gamma$, if one of the coefficients is $\pm 1$.  The coefficient of $e_{ij}$ is $\pm 1$ precisely when $\lcm(f_i,f_j,f_k) = \lcm(f_i,f_j)$ for some $k$.  When this happens, the syzygy $\psi(e_{ij})$ can be written as a combination of other syzygies $\psi(e_{jk}), \psi(e_{ik})$. 

We call a syzygy $\psi(e_{ij})$ {\bf minimal} if for each $k\notin \{i,j\}$, $\lcm(f_i,f_j) \neq \lcm(f_i,f_j,f_k)$. Otherwise, we call the syzygy {\bf non-minimal}. 

\begin{Remark}
	The fact that the Taylor resolution is exact implies that the ${b \choose 2}$ syzygies in the Taylor resolution fully generate $\ker \phi$.  If there are non-minimal syzygies then some of these syzygies are redundant and may be removed.  In general, one must take care when doing this. For instance, if $M = (xy,yz,xz)$ then {\bf all} of the Taylor syzygies are non-minimal. Obviously one cannot remove all three.  This type of behavior will not happen for the ideals in this paper and the following technical lemma gives a condition that will suffice for our purposes. 

	\end{Remark}

\begin{Lemma}\label{lemma10}
Let $M$ be a monomial ideal. Suppose that each non-minimal Taylor syzygy is in the span of the minimal Taylor syzygies.  Then the minimal syzygies span all of $\ker \phi$.  In this case, $\beta_2(R/M)$ is equal to the number of minimal Taylor syzygies. 
\end{Lemma}

\begin{proof}
Since the complete set of Taylor syzygies generates $\ker \phi$, the first claim is clear.  The second follows from the fact that none of the minimal Taylor syzygies can be written as a combination of the others.  Indeed, suppose without a loss of generality that $\psi(e_{12}) = \sum_{i<j} c_{ij} \psi(e_{ij})$ where $c_{12} = 0$. Then the element $e_{12} - \sum_{i<j} c_{ij}e_{ij}$ (whose coefficient of $e_{12}$ is 1) is in the kernel of $\psi$.  But by exactness this means that 
$$e_{12} - \sum_{i<j} c_{ij}e_{ij} = 
\sum_{i<j<k} {\alpha}_{ijk} \gamma(e_{ijk})
$$
but this is impossible since the coefficients of $e_{12}$ in $\gamma(e_{12k})$ will always be in the maximal ideal since $\psi(e_{12})$ is a minimal syzygy. 
\end{proof}

We will call a syzygy on $f_i,f_j$ {\bf linear}, (respectively {\bf quadratic)} if the degree of $\lcm(f_i,f_j)$ is one (respectively two) larger than $\max(\deg f_i, \deg f_j)$.   We say the {\bf degree} of the syzygy is $\deg \lcm(f_i,f_j)$. For instance, the Taylor syzygy on the pair $(xyz, xywuv)$ is linear of degree 6 and the Taylor syzygy on $(x^2, yzw)$ is quadratic of degree 5. 

\begin{Example}
Before we prove the general case, we show what the Taylor resolution says about a licci ideal generated in degree $2$, continuing Example \ref{ex6} by showing that $(3) \implies (2)$ in Theorem \ref{MainThm}.	We know that 
$$I = Q + {a}(y_1, \ldots, y_n).$$
We will show that $\beta_{23}(R/I) = {n\choose 2} + 2n$.  Note that the only syzygy pairs that can contribute to $\beta_{23}(R/I)$ would be linear syzygies, where the generators share a common factor.  There are ${n\choose 2}$ of type $(ay_i,ay_j)$ and exactly $2n$ of type $(ay_i,x^2)$ (where $x$ is one of the factors of $ay_i$). It is straightforward to check that these are all minimal syzygies and so we are done by Lemma \ref{lemma10}.
\end{Example}

The following proposition will complete the proof that $(3) \implies (2)$ in Theorem \ref{MainThm}.
\begin{Proposition}\label{Prop12}
	Suppose that $I = Q + (G)$ is an ideal of the form (3) in Theorem \ref{MainThm}.  That is, there are distinct variables $a_i,y_{\ell m}$ so that the degree $j$ generators of $G$ are of the form 
	$$a_1\cdots a_{j-1}(y_{j1}, y_{j2}, y_{j3}, \ldots, y_{jb_j}).$$
The non-minimal Taylor syzygies are all in the span of the minimal ones. Thus 
$$\beta_{2j}(R/I)=\mbox{ the number of minimal Taylor syzygies of degree $j$}.$$
The formulas in (2) of Theorem \ref{MainThm} are these numbers. 
\end{Proposition}
\begin{proof}
Let $f,h$ be two monomial generators of $I$ of degrees $i$ and $j$ respectively with $i \leq j$.   

\begin{enumerate} \item If both are squarefree, then 
$$f = a_1\cdots a_{i-1}y, \ \ 
h = a_1\cdots a_{j-1}z,
$$
for some distinct variables $y,z$.  The lcm of $f$ and $h$ is $hy$ and has degree $j+1$.  Note that the Taylor syzygy of this pair will always be minimal, because if $h'$ is any third generator, then either $h'$ is a square or $h'$ is a squarefree generator that will contain a new variable $y_{pq}$.  In either case, $\deg \lcm(f,h,h')\geq j+2$. The total number of such pairs $(f,h)$ is ${b_j \choose 2} + b_j(b_2 + \cdots + b_{j-1})$, which all contribute to $\beta_{2,j+1}$.
\item If $f$ is a square and $h$ is squarefree, then $i=2$ and there are two cases: 
	\begin{enumerate}
	\item {\bf Linear Syzygy} $(x^2,h)$: meaning that the squarefree generator has a factor of $x$.  These are minimal since adding another generator must either add a new variable $y_{pq}$ not appearing in $h$ or be a new square, thus increasing the degree.  There are $b_j$ choices for $h$ and $j$ choices for the factor $x$, so there are $jb_j$ such minimal syzygies that contribute to $\beta_{2,j+1}$.
	\item {\bf Quadratic Syzygy} $\{x^2,h\}$: meaning that the squarefree generator does not have $x$ in it. There are $b_j$ choices for $h$ and $(g-j)$ choices for $x$. Of these $b_j(g-j)$ syzygies, some can be non-minimal. Suppose $h = a_1\cdots {a_{j-1}}z$. Let us consider $\lcm(x^2,h,h')$ as $h'$ ranges over the generators of $I$, with an aim to detect non-minimality.  If $h'$ is a square, then clearly this $\lcm$ will not equal $\lcm(x^2,h)$.  If $h'$ is squarefree, then $h'$ must necessarily involve a new variable $y_{pq}$ not in $h$.  So the only way for $\lcm(x^2,h,h') = \lcm(x^2,h)$ is if $x = y_{pq}$ and $h' = {a_1}{a_2}\cdots a_{k-1}y_{pq}$.  Note that this forces $k\leq j$, since otherwise $a_{k-1}$ would not be present in $x^2$ or $h$.  Thus the only non-minimal pairs of this type must be of the form $(y_{pq}^2,h)$, where $\deg h\geq p$.  These non-minimal syzygies $(y_{pq}^2,h)$ will be in the span of the minimal (by 2(a) and (1))  syzygies $(y_{pq}^2, h')$ and $(h,h')$ and so by Lemma \ref{lemma10} to find the contribution to the Betti table we can take the total number of pairs and subtract the number of non-minimal ones, for which there are $b_j$ choices for $h$ and $(b_j-1)+b_{j-1}+\cdots+b_2$ choices for $y_{pq}$.  Since these syzygies are quadratic, they contribute to $\beta_{2,j+2}$ an amount of 
	$$b_j(g-j)-b_j((b_j-1)+b_{j-1}+\cdots+b_2).$$
	After simplifying and shifting indices, the contribution to $\beta_{2,j+1}$ will be
	$$b_{j-1}(g- b_2 - b_3 -\cdots - b_{j-1}-(j-2)).$$

	\end{enumerate}
\item If both are squares then $i = j = 2$.  The quadratic syzygy on $(x^2,y^2)$ will be non-minimal if and only if there is no generator $h$ such that $\lcm(x^2,y^2,h) = x^2y^2$.  Clearly this can only happen if $h = xy$ is one of the $b_2$ quadratic generators.  Such non-minimal syzygies $\{x^2,y^2\}$ will be in the span of the minimal (by 2a) syzygies $\{x^2,xy\}$ and $\{y^2,xy\}$. There are ${g \choose 2}$ pairs $(x^2,y^2)$ and thus the number of minimal syzygies will be ${g \choose 2} - b_2$. This number contributes to $\beta_{24}$. \qedhere
\end{enumerate} 
\end{proof}

\section{Proof that $(2)\implies (3)$}

Finally, we are ready to complete the proof of Theorem \ref{MainThm}.  We begin with a continuation of our running example and then some preliminaries. 

\begin{Example}We will continue Example \ref{ex6} and show that the structure can be detected only from the Betti number $\beta_{23}$.  Suppose that $I = Q + (f_1, \ldots, f_n)$ is an ideal generated by $n$ squarefree quadrics and the squares of the variables and $\beta_{23} = {n \choose 2} + 2n$. The only Taylor syzygies that could possibly be linear are the pairs $(x^2,f_i)$ where $x$ divides $f_i$ and the pairs $(f_i,f_j)$ where $\gcd(f_i,f_j) \neq 1$.  There are at most $2n$ of the first type and ${n \choose 2}$ of the second.  Thus it must happen that {\bf all} of the pairs $(f_i,f_j)$ must be minimal and linear.  If $n = 1$ then $I = Q + (ay)$ for some variables $a,y$, as expected.  If $n \geq 1$ then let $a = \gcd(f_1,f_2)$, so $f_1 = ay_1$ and $f_2 = ay_2$. We will show that $a$ is a factor of each $f_j$.  If $n = 2$ we are done.  If $n\geq 3$, suppose without a loss of generality that $f_3$ is not divisible by $a$.  But $(ay_1,f_3)$ and $(ay_2,f_3)$ are both linear syzygies so $f_3$ must be $y_1y_2$, but then $\lcm(f_1,f_2,f_3) = \lcm(f_1,f_2)$, contradicting minimality.
\end{Example}

\begin{Proposition}\label{Prop13}
	Suppose that $J$ is a squarefree monomial ideal generated in a single degree $d$ with $\mu$ generators.  
Suppose that $J$ has ${\mu \choose 2}$ minimal linear syzygies. Then $J = D(y_1, \ldots, y_\mu)$ where the $y_i$ are variables and $D$ is a monomial of degree $d-1$.    Hence by Algorithm \ref{alg1}, the ideal $Q + J$ is licci. 
\end{Proposition}
\begin{proof}
Note that the claim is clear if $\mu = 1$, so suppose $\mu \geq 2$.  Since the Taylor resolution is exact, a monomial ideal with $\mu$ generators has at most ${\mu \choose 2}$ total Taylor syzygies, and so the condition in the statement means that all of these syzygies are minimal.  The linearity means that any two monomial generators have a gcd of degree $d-1$.  Take two generators $m_1, m_2$ and let $D$ denote their gcd.  Then we have $m_1 = Dy_1, m_2 = Dy_2$. If $\mu = 2$ we are done.  Suppose $\mu \geq 3$.  We claim that any other generator of $J$ will be divisible by $D$. This will complete the proof. We will argue by contradiction. Suppose there is some $m_3$ that is not divisible by $D$. However, since the syzygies with $m_1, m_2$ must be linear, $m_3$ must share a degree $d-1$ factor with each of $Dy_1$ and $Dy_2$. Evidently then $m_3$ must be divisible by $y_1y_2$.  For degree reasons, the remaining $d-2$ factors of $m_3$ must be factors of $D$.  But then
$$\gcd(m_1,m_2,m_3) = \gcd(m_1,m_2)$$
since $m_3$ introduces nothing new. Therefore the Taylor syzygies are not all minimal, contradicting our assumption.
\end{proof}
 
\begin{Proposition}\label{Prop:QJsingledegree}
	Suppose that $J$ is a squarefree monomial ideal generated in a single degree $d\geq 2$ and $I = Q + J$ has Betti numbers equal to those in the table in (2) of Theorem \ref{MainThm}. Then $I$ is licci.
\end{Proposition} 
\begin{proof}
Suppose that $J$ has $b_d$ generators of degree $d \geq 2$. 	By Proposition \ref{Prop13} it is sufficient to show that $J$ has ${b_d \choose 2}$ minimal linear syzygies. 

\begin{enumerate}
\item Case 1 $(d =2)$: $I$ is minimally generated by $g+b_2$ quadrics. The table says that $\beta_{23}(R/I) = {b_2\choose 2} + 2b_2$. The only possible pairs of generators that can give a Taylor syzygy that contributes to $\beta_{23}$ will be pairs $(x^2,h)$ with $h\in J$ and $x|h$ or pairs $(h,h')$ of squarefree monomials in $J$. There are exactly $2b_2 + {b_2\choose 2}$ such syzygies and thus they must all be minimal.  Thus $J$ has ${b_2\choose 2}$ linear syzygies as required. 
\item Case 2 $(d=3)$: $I$ is generated by $g$ squares and $b_3$ cubic squarefree monomials. The table says that $\beta_{24} = {b_3\choose 2}  + {g \choose 2} + 3b_3$.  The only possible Taylor pairs that can contribute to $\beta_{24}$ will be the ${g \choose 2}$ pairs of squares, the $3b_3$ pairs $(x^2, h)$ where $x|h$ and $h\in J$, and the ${b_3\choose 2}$ pairs of elements in $J$ that give linear syzygies. Thus all of these must in fact be minimal and so $J$ must itself have ${b_3\choose 2}$ linear syzygies and we are done.
\item Case 3 ($d \geq 4$): The table says that $\beta_{d,d+1}={b_d\choose 2} + db_d$ linear syzygies. Arguing as in Case 2, we must have that $J$ has ${b_d\choose 2}$ linear syzygies and we are done.\qedhere 
\end{enumerate}
\end{proof}
The following lemma is almost harder to carefully state than it is to prove.  The point is that if one knows inductively that up to a certain degree the generators of an ideal have the structure described in (3) of Theorem \ref{MainThm}, then the following is a sufficient condition to ensure that the in the next degree the pattern continues. 
\begin{Lemma}\label{lemma:addlastdeg}
Suppose that $J$ is an ideal whose generators are of the form: 
\begin{eqnarray*}
	\mbox{degree $2$}:& & a_1(y_{11}, \ldots, y_{1b_2}),\\	
	\mbox{degree $3$}:& & a_1a_2(y_1, \ldots, y_{2b_3}),\\	
	& & \vdots \\
	\mbox{degree $e$}:& & a_1a_2\cdots a_{e-1}(y_{e1},\ldots, y_{eb_e})
\end{eqnarray*}
where all of the variables are distinct. Assume that $b_e>0$ (i.e. there is at least one generator of degree $e$) but otherwise allow the $b_j$ to be possibly $0$.  

Suppose that $Dz_1, \ldots, Dz_n$ are distinct squarefree degree $r$ monomials with $r>e, n\geq 1$, for some variables $z_j$.  We do not assume that the variables in $Dz_j$ are distinct from those above, but we do assume that no $Dz_i$ is in the ideal $J$.  Suppose that each syzygy pair $(h,Dz_j)$ for $h$ a generator of $J$ is linear and minimal. Then up to relabeling, the degree $r$ generators are of the form:
\begin{eqnarray*}
	\mbox{degree $r$}:& & a_1a_2\cdots a_{r-1}(y_{r1},\ldots, y_{rb_n})
\end{eqnarray*}
where $a_{e},\ldots, a_{r-1}, y_{rj}$ are new distinct variables.   In particular the ideal $I = Q+J+D(z_1,\ldots, z_n)$ is licci, by Theorem \ref{MainThm}.
\end{Lemma}

\begin{proof}
Consider the degree $e$ generator $h = a_1a_2\cdots a_{e-1}y_{e1}$.  Suppose that $Dz_1$ does not contain some factor $a_j$ in $h$. Then since the syzygy $(Dz_1,h)$ is linear, we must have that $\lcm(Dz_1,h) = Dz_1a_j$ and can deduce that $Dz_1$ must in fact have every other variable of $h$ as a factor. 

\begin{itemize}
\item Case 1: ($a_j$ is a factor of some other minimal generator $h'$ of $J$).  Again by linearity we must have that $Dz_1$ is divisible by a factor of $h'$ except for $a_j$.  But then $\lcm(Dz_1,h,h') = \lcm(Dz_1,h')$, contradicting minimality. Thus $Dz_1$ must contain each $a_j$ as a factor in this case. 
\item Case 2: ($a_j$ only appears in the single generator $h$).  This means that $b_e =1$ and that there are not generators of $J$ of degrees $j+1, j+2, \ldots, e-1$.  But then we can relabel the single generator $h$ as 
$$h = (a_1a_2\cdots \widehat{a_j}\cdots a_{e-1}y_{e1})(a_j)$$
thinking of the first $e-1$ variables as the new $a_i$'s and $a_j$ as $y_{e1}$, and now $Dz_1$ does contain each $a_j$ as a factor.
\end{itemize}
Thus we can assume that $Dz_1$ contains each of $a_1,\ldots, a_{e-1}$ as a factor, that is $Dz_1 = a_1\cdots a_{e-1}M$ for some monomial $M$.  Note that $M$ cannot be divisible by any variable present in the minimal generators of $J$ (since then $Dz_1$ would not be a minimal generator).  

If $Dz_1$ is the only generator of degree $r$ we are done, and can simply declare the remaining factors of $M$ (save 1, which will be $y_{r1}$) to be $a_e, \ldots, a_{r-1}$.  

If $n\geq 2$ we will show that $D$ is divisible by $a_1\cdots a_{e-1}$.  Suppose that some $a_\ell$ with $\ell \leq e-1$ is not a factor of $D$.  Then since $Dz_1$ is divisible by $a_\ell$ we must have that $z_1 = a_\ell$.  

Now consider the generator $Dz_2$. This cannot have a factor of $a_\ell$, since $a_\ell$ is not a factor of $D$ and $z_2 \neq z_1$.  Then since the syzygy $(Dz_2,h)$ is linear, we must have $$\lcm(Dz_2,h) = Dz_2a_\ell = Dz_2z_1 = \lcm (Dz_2,h, Dz_1),$$ which contradicts minimality. 

Thus $a_1\cdots a_{e-1}$ divides $D$. We can declare the remaining factors of $D$ to be $a_e,\ldots, a_{r-1}$ and then set $z_1, \ldots, z_{b_n}$ to be $y_{r1},\ldots, y_{rb_n}$.\qedhere
\end{proof}
The following proposition completes the proof that $(2) \implies (3)$ in Theorem \ref{MainThm}.
\begin{Proposition}\label{Prop11}
Suppose that $I$ is a monomial ideal that contains the squares of the variables. Suppose that the Betti numbers of $I$ are those in the table in part (2) of Theorem \ref{MainThm}. Then $I$ is licci.  
\end{Proposition}

\begin{proof}
Let $G$ denote the set of squarefree generators. Suppose that $I = Q + (G)$ and that the Betti numbers of $I$ match those in the table.  We will prove that $I$ is licci.  We will induct on $|G|$.  If $|G| = 1$ then $I$ is licci perforce and we are done. Suppose that the claim is true for all ideals with at most $b$ generators, now suppose that $|G| = b+1$. If the elements of $G$ all have the same degree, $I$ is licci by Proposition \ref{Prop:QJsingledegree}.  

Suppose that $G$ has elements in at least two distinct degrees.  Let $G_1$ denote the set of all squarefree generators of submaximal degree, that is, $G = G_1 \cup G_r$ where $G_r$ consists of all the generators of maximal degree, $r$.  Note that $r\geq 3$.  We will consider the ideal $K = Q + (G_1)$.  Note that until the row containing $\beta_{1r}$, the Betti tables of $K$ and $J$ are identical, since those rows cannot depend on generators of degree $r$.  Thus, the Betti numbers of $K$ match those of the table and so by induction $K$ is licci and thus by Theorem \ref{Thm5} these squarefree generators of degree $j$ have the form $a_1\cdots {a_{j-1}}y_{j\ell}$ for distinct variables $a_i,y_{pq}$. 

Suppose that $e$ is the largest degree of a generator of $K$. Then to complete the proof we need to show that 
\begin{equation*}\begin{split}
\mbox{There exist (new) variables $a_e, \ldots, a_{r-1}, y_{r1}, \ldots, y_{rb_r}$, so that }\\ \mbox{every degree $r$ generator has the form $a_1\cdots a_{r-1}y_{j\ell}$}.
\end{split}
\tag{$\star$}	
\end{equation*}

When we say ``new'' we mean that these variables do not appear in any squarefree generator of $K$.  We will show that we can apply Lemma \ref{lemma:addlastdeg}. We are given that $\beta_{2,r+1}$ is the sum of the following:

\begin{enumerate}
\item ${b_r \choose 2}$
\item  $b_r(b_2+\cdots+b_{r-1})$
\item  $rb_r$
\item $(g-(b_2+ \cdots +b_{r-1}+(r-2)))b_{r-1}$
\item (if $r = 3$) ${g \choose 2} - b_2$
\end{enumerate}

To complete the proof, we must show that $I$ has the form given in part (3) of Theorem \ref{MainThm}.
Now the only difference between $K$ and $I$ are the $b_r$ generators of degree $r$.  Let us carefully consider all possible syzygy pairs $(f,h)$ that could contribute to $\beta_{2,r+1}$.  Note that $f$ and $h$ must be chosen from: 
\begin{enumerate}
	\item[$Q$.] A square $x^2$ from $Q$
	\item[$K$.] A squarefree generator $a_1\cdots {a_{j-1}}y_{j\ell}$ from $K$
	\item[$r$.] A ``new'' squarefree generator of degree $r$.
\end{enumerate}

Thus there are six cases to consider.
\begin{enumerate}
\item[($QQ$)] These will only contribute to $\beta_{2,r+1}$ if $r=3$, in which case the number of minimal such pairs will be ${g\choose 2} -b_2$, which is item $(5)$ in the list above. 
\item[($QK$)] For the degrees to match, this must be a quadratic syzygy of a square $x^2$ in $Q$ with a squarefree monomial $h=a_1\cdots {a_{r-1}}y_{j\ell}$ in $K$, where $x$ is not a factor of $h$.  There are $b_{r-1}(g-(r-1))$ such pairs and by the discussion in 2b of Proposition \ref{Prop12} at least $((b_{r-1}-1)+b_{r-2}+\cdots+b_2)b_{r-1}$ of them are non-minimal (their nonminimality come from the generators of $K$).  It is of course possible that the new generators of degree $r$ could deem others of these to be non-minimal.  But in any event we have at most: $(g-(b_2+ \cdots +b_{r-1}+(r-2)))b_{r-1}$ minimal syzygy pairs of this type contributing to $\beta_{2,r+1}$, which is item $(4)$ in the list above. 
\item[($KK$)] For degree reasons, none of these pairs can contribute to $\beta_{2,r+1}$. 
\item[($Qr$)] For a pair $(x^2,h)$ to give a syzygy of degree $r+1$ it must be the case that $x$ is a factor of $h$.  Minimal or not, there are $b_r$ choices for $h$ and $r$ choices for $x$.  This is item $(3)$ above. 
\item[($Kr$)] Without even considering the degrees of the syzygies, there are at most $b_r(b_2 + \cdots b_{r-1})$ pairs $(g,h)$ of this form. This is item $(2)$ above.
\item[$(rr)$] Again, without even considering the degrees, there are at most ${b_r \choose 2}$ such pairs. This is item $(1)$ above. 
\end{enumerate}
The 6 upper bounds just established are exactly equal to the 6 terms of the sum equal to $\beta_{2,r+1}(R/I)$.  Thus equality must hold for each type.  In particular, all of the ${b_r\choose 2}$ syzygies of type $(rr)$ must in fact be linear syzygies. Thus by Proposition \ref{Prop13} we have that the generators of $J$ are of the form $Dz_1, \ldots, Dz_{b_r}$ for some monomial $D$ and distinct variables $z_i$.  

Now consider the $b_r(b_2+\cdots + b_{r-1})$ syzygies of type $(Kr)$.  These must all be linear and minimal as well, are now in the situation to apply Lemma \ref{lemma:addlastdeg} which completes the proof.
\end{proof}

\section*{Acknowledgements}
Support for this research was provided
by an AMS-Simons Research Enhancement Grant for Primarily Undergraduate Institution Faculty.  The author thanks 
David Eisenbud, Bennet Goeckner, Elo\'isa Grifo, Craig Huneke, Srikanth Iyengar, Claudia Polini, and Bernd Ulrich for helpful conversations and correspondence. 

\end{document}